\documentclass[10pt,reqno]{amsart}
\usepackage{amsmath, latexsym, amsfonts, amssymb,amsthm, amscd,epsfig,enumerate}
\usepackage{subfigure}
\advance\hoffset -.75cm

\usepackage{amsfonts}
\usepackage{latexsym}
\usepackage{amsmath}
\usepackage{amssymb}
\usepackage[usesnames]{xcolor}

\definecolor{refkey}{gray}{.75}
\definecolor{labelkey}{gray}{.75}

\newcommand{\N}{\mathbb N}
\newcommand{\pr}{\mathbb P}
\newcommand{\E}{\mathbb E}

\newcommand{\diff}{{\rm d}}

\newcommand{\sgn}{\mathop{\mathrm{sgn}}}

\newcommand{\ident}{{\mathchoice {\rm 1\mskip-4mu l} {\rm 1\mskip-4mu l}
{\rm 1\mskip-4.5mu l} {\rm 1\mskip-5mu l}}}

\renewcommand{\thesubfigure}{\arabic{subfigure}}
\makeatletter
\renewcommand{\@thesubfigure}{\tiny Figure \thesubfigure: \space}
\renewcommand{\p@subfigure}{}
\makeatother

\newtheorem{teo}{Theorem}[section]
\newtheorem{lem}[teo]{Lemma}

\newtheorem{rem}[teo]{Remark}

\newtheorem{exmp}[teo]{Example}

\begin{document}

\title[Persistent and susceptible bacteria with individual deaths]
{Persistent and susceptible bacteria with individual deaths}

\author[F.~Zucca]{Fabio Zucca}
\address{F.~Zucca, Dipartimento di Matematica,
Politecnico di Milano,
Piazza Leonardo da Vinci 32, 20133 Milano, Italy.}
\email{fabio.zucca\@@polimi.it}

\begin{abstract}
The aim of this paper is to study two models for a bacterial population subject to
antibiotic treatments.
It is known that some bacteria are sensitive to antibiotics. These bacteria are in a
state called \textit{persistence} and each bacterium can switch from this state to
a non-persistent (or susceptible) state and back.
Our models extend those introduced in
\cite{cf:GaMaSch} by adding a (random) natural life cycle for each bacterium and
by allowing bacteria in the susceptible state to escape
the action of the antibiotics with a fixed probability
$1-p$ (while every bacterium in a persistent state survives with
probability $1$). In the first model we ``inject'' the antibiotics in the system at fixed,
deterministic times while in the second one the time intervals are random.
We show that, in order to kill eventually the whole bacterial population, these
time intervals cannot be ``too large''. The maximum admissible length is increasing with respect to
$p$ and it decreases rapidly when $p<1$. 
While in the case $p=1$ switching back and forth to the persistent state is the only chance of surviving
for bacteria, when $p<1$ and the death rate in the persistent case is positive then switching state is
not always a good strategy from the bacteria point of view.
\end{abstract}

\date{}
\maketitle


\baselineskip .6 cm

\section{Introduction}
\label{sec:intro}
\setcounter{equation}{0}

It is well known that some bacteria are not sensitive to antibiotics (see \cite{cf:Bigger44}).
This state, called \textit{persistence}, is not permanent and each bacterium can switch during
its lifetime from persistent to susceptible and back to persistent many times
(see for instance \cite{cf:KKBL05, cf:Levin04}).
In the persistent state it does not reproduce, while in the susceptible state it breeds but it
is also vulnerable to antibiotics.

Two models for this phenomenon have been introduced in \cite{cf:GaMaSch}.
The aim of this paper is to extend these  two models by adding (1) a life cycle for each bacterium
(that is, individual deaths) and (2) a possibly positive probability $1-p$ for each
bacterium in the susceptible state to survive the action of the antibiotic.
In our models each bacterium has an independent random lifetime represented by
two exponentially distributed random variables with parameters $d_n$ and $d_r$
for the susceptible state and the persistent state respectively. At certain times, that we
call \textit{mass killing times} or simply \textit{killing times}, an antibiotic is injected in the system; the time intervals are
deterministic and equally spaced in the first model and random in the second one.
The action of the antibiotic does not affect the persistent population
but it kills each bacterium in the susceptible state independently with probability $p \in [0,1]$;
$p=0$ means that there is no target for the antibiotic in the bacterial genome,
$p=1$ means that the antibiotic performs a ``perfect'' mass-killing action in the susceptible state population.
The models in \cite{cf:GaMaSch} can be recovered by setting $d_n=d_r=0$ and $p=1$.

For some values of the parameters (see Section~\ref{subsec:meanfield} for details),
the system dies out almost surely even without the action of the antibiotics; thus,
we need to study only the so-called supercritical case, which is the case when
the natural evolution of the system allows survival with positive probability.

In the deterministic killing times case we suppose that the mass-killings occur at times
$S_n:=n T$ where $T>0$.
We show that, for every value of the parameters, if the interval $T$ between each killing time is too large
(strictly larger than a critical value $T_c \in (0,+\infty)$), the bacterial population has a positive
probability of survival, while if $T \le T_c$ there is almost sure extinction (see Theorem~\ref{th:tcritical}). 
In particular we are interested in the dependence of $T_c$ from $p$: as
$p$ converges to $0$, the critical time interval length $T_c$ converges to $0$ as well.
Thus there is not a positive minimum time interval which guarantees the extinction of
the bacterial population for all $p\in (0,1]$. When the death rate $d_r$ is positive,
for some set of parameters it might happen that switching from the susceptible to the resistant state
is not a good strategy from the bacteria point of view, since it results in a longer critical time
$T_c$. 

In the random killing time case we suppose that the mass-killings are separated by a sequence
of independent random time intervals $\{T_n\}_{n \ge 1}$; this sequence is i.i.d.~and the distribution
is given by a probability measure $\mu_\beta$, where $\{\mu_\beta\}_{\beta >0}$
is a one-parameter stochastically increasing family of probability measures satisfying
some mild conditions (see Section~\ref{sec:random} for details).
In this case we have two randomizations, so to speak: first we choose a realization $\xi$ of the
sequence $\{T_n\}_{n \ge 1}$ (we call $\xi$ a \textit{realization of the environment}) 
and then we have a random evolution of the system with killing times given by $\xi$.
We show that if $\beta$ is  large enough ($\beta > \beta_c^1(p)$) the population survives with positive probability
for almost every realization of the environment (see Theorem~\ref{th:survivalRT}).
On the other hands if $\beta$ is small enough ($\beta < \beta_c^2(p)$) then the population
dies out almost surely for almost every realization of the environment (see Theorem~\ref{th:extinctionRT}). 
As in the deterministic case,
$\lim_{p \to 0} \beta_c^2(p)=0$. Roughly speaking,
since the expected time between two
consecutive mass-killings is a nondecreasing function of $\beta$, we have that,
in order to kill almost surely the bacterial population, the expected time between
two injection of antibiotics in the system cannot be too large.
According to Example~\ref{exmp:ex1},
it might happen that $ \beta_c^2(p)< \beta_c^1(p)$.
In particular, if $\mu_\beta \sim \mathrm{Exp}(1/\beta)$ (where $\mathrm{Exp}$ is the exponential
distribution) then $\beta^1_c(p)=\beta_c^2(p)=:\beta_c(p)$
and
 $\lim_{p \to 0} \beta_c(p)=0$.

\section{The dynamics}
\label{sec:discrete-continuous}

This is a modification of the model described in \cite{cf:GaMaSch}
with the introduction of individual deaths for each type of
bacteria; indeed, it is quite natural to assume that each bacterium has
its own life cycle in the absence of an antibiotic treatment.
Another addition to the dynamics is the possibility for some
susceptible bacteria to survive (with a fixed probability $1-p$ where $p \in [0,1]$) the
action of the antibiotics.
We denote by $N_t$ and $R_r$ the number of susceptible and persistent bacteria
respectively.
This is a 2-type process in continuous time, with the following (nonnegative)
rates:
\begin{equation} \label{eq:rates}
 \begin{matrix}
   (N_t,R_t) &\to & (N_t+1,R_t)  & \textrm{at rate } \lambda N_t \\
 (N_t,R_t) &\to & (N_t-1,R_t+1)  & \textrm{at rate } a N_t \\
  (N_t,R_t) &\to & (N_t+1,R_t-1)  & \textrm{at rate } b R_t \\
  (N_t,R_t) &\to & (N_t-1,R_t)  & \textrm{at rate } d_n N_t \\
  (N_t,R_t) &\to & (N_t,R_t-1)  & \textrm{at rate } d_r R_t. \\
 \end{matrix}
\end{equation}
We recall that a change of state takes place at rate $\alpha$ if it takes place
after a random exponentially distributed time intervals $T \sim \mathrm{Exp}(\alpha)$:
due to the lack of memory of the exponential distribution, this means that whenever we start looking
at the system, the random time to wait before the change of state is a $\mathrm{Exp}(\alpha)$-distributed random variable.
In particular the probability of the change of state in an interval of time $[t,t+\Delta t]$ is
asymptotic to $\alpha \cdot \Delta t$ as $\Delta t$ goes to $0$.
A more precise construction of the model is given in the proof of Theorem~\ref{th:tcriticalrandom}.
Roughly speaking, we can imagine that each particle has five clocks which ring at
exponentially distributed time intervals with parameters $\lambda$, $a$, $b$, $d_n$ and $d_r$
(the clocks are independent).
When a particle is in a susceptible state we have different possibilities: if its $\mathrm{Exp}(\lambda)$-clock rings it breeds,
if its $\mathrm{Exp}(d_n)$-clock rings it dies and if its $\mathrm{Exp}(a)$-clock rings it changes into a persistent state
(it is not affected by the other clocks).
On the other hand when a particle is in a persistent state we observe the following behaviors:
if its $\mathrm{Exp}(d_r)$-clock rings it dies and if its $\mathrm{Exp}(b)$-clock rings it moves to a susceptible state
(and, again, it is not affected by the other clocks).

When $a=0$ and $b=0$ the two populations are completely separated,
the $N$-population is a branching process with mass killing and the $R$-population is
stable if $d_r=0$ or dying out if $d_r>0$.
If $a>0$ and $b=0$ then  the $N$-population is a branching process with mass killing and
individual death rate $a+d_n$, while the $R$-population 
survives
if and only if either $d_r=0$ or the $N$-population survives.
The interesting case is $b>0$.
We note that this process is not monotone with respect to the parameters $a$ and $b$;
on the other hand, it is monotone with respect to the other parameters and to the initial condition.

Without the mass deaths caused by the antibiotics, the process has a
discrete-time branching random walk counterpart (similar to the one described in \cite{cf:Z1, cf:BZ4}).
When the antibiotic is injected in the system the dynamics is the following
\[
 (N_t,R_t) \to (\, B(N_t, 1-p)\, , R_t)
\]
which means that the number of surviving susceptible bacteria is a
binomial-distributed random variable; thus, at a killing
time each  susceptible bacterium is killed (independently from the others)
with probability $p \in [0,1]$.
After a mass killing 
the system performs a new evolution starting
from the survivors. If we consider just the surviving population
at these mass killing times,
we have a discrete-time process; it turns out to be a 2-type branching process
or a 2-type branching process in random environment depending on our choice
of the killing times (deterministic or random). Our choice will be either an increasing sequence of killing times
$\{S_n\}_{n \ge 1}$ where $S_n=n T$ (for a fixed $T>0$) or $S_{n}=\sum_{i=1}^n T_n$ where $\{T_n\}_{n \ge 1}$ is an
i.i.d.~sequence ($S_0:=0$).

According to \cite{cf:AthNey, cf:Harris63, cf:BZ4} the long-term behavior
of this discrete-time branching process depends only of its first-moment matrix $M=(m_{ij})_{i,j=1,2}$.
where $m_{ij}$ is the expected number of offsprings of type $j$ from a particle of type $i$
(see for instance \cite{cf:Z1, cf:BZ4}). In order to compute $M$ we need
to consider the \textit{mean field model} (this is done in Secton~\ref{subsec:meanfield}). 
The main results on the deterministic case and the random case are in Sections~\ref{sec:determ} and \ref{sec:random}
respectively.
We note that all these results  hold for any finite (non-void)
initial condition.
All the proofs and technical Lemmas can be found in Section~\ref{sec:proofs}. 

\subsection{Mean field model}\label{subsec:meanfield}

This section is a useful exercise which allows us to obtain
some explicit expressions the we need in the sequel.
The linear system of equations for the expected values
$(n_t, r_t):=\E[(N_t,R_t)]$ is
\begin{equation}\label{eq:systdiff}
 \begin{cases}
\frac{\diff}{\diff t} n_t =(\lambda-a-d_n) n_t +b r_t  \\
\frac{\diff}{\diff t} r_t =a n_t -(b+d_r) r_t,  \\
\end{cases}
\end{equation}
where $b>0$ and $\lambda,a, d_n, d_r \ge0$.
The eigenvalues $x^+$, $x^-$ (where $x^+ \ge x^-$) of the corresponding matrix
\[
 A:=\begin{pmatrix}
  \lambda-a-d_n & b \\
a & -(b+d_r) \\
 \end{pmatrix}
\]
are the solutions of the equation
\begin{equation}\label{eq:systdiff2}
h(x):=x^2+x(b+d_r-\lambda+a+d_n)-((b+d_r)(\lambda-d_n)-ad_r)=0.
\end{equation}
We note immediately that, since $h(-(b+d_r))=h(\lambda-a-d_n)=-ab \le 0$, then
the eigenvalues are always real numbers and
$x^- \le \min(-(b+d_r),\lambda-a-d_n)$ 
and $x^+ \ge \max(-(b+d_r),\lambda-a-d_n)$
(one of equalities holds if and only if both hold, that is, if and only if
$ab=0$). 
Moreover, the basic branching process theory tells us that if the maximum eigenvalue $x^+ \le 0$
then we have almost sure (spontaneous) extinction.
Hence if the determinant of the matrix $h(0)\equiv -(b+d_r)(\lambda-d_n)+ad_r \ge 0$ we have extinction for all $p$ and for
any choice of $\{T_n\}_{n \ge 1}$ (even when $T_1=+\infty$).
%
%
From now on we assume
\begin{equation}\label{eq:assumpdet}
 (b+d_r)(\lambda-d_n)-ad_r > 0,
\end{equation}
that implies immediately
$x^+>0$; 
hence $x^+>x^-$. A corresponding pairs of eigenvectors is
$Z=(1,a/(b+d_r+x^+))$, $C=(1, a/(b+d_r+x^+))$.
%
%
The generic solution can be written as
\[
\begin{pmatrix}
n(t) \\
r(t)
\end{pmatrix}=
e^{At}
\begin{pmatrix}
n(0) \\
r(0)
\end{pmatrix}
\]
where $e^B:= \sum_{i=0}^\infty B^i/i!$ for every matrix $B$ and $(n(0),r(0))$ is the initial state.
The explicit computations of $e^{At}$ are easy: 
one simply needs to evaluate the solution of the system starting from $(1,0)$ and $(0,1)$.
Note that $x^++x^-=\lambda-b-d_r-d_n-a$ and that $x^+x^-=ad_r-(b+d_r)(\lambda-d_n)$.
We have
\begin{equation}\label{eq:fundsolutions}
\begin{split}
 \begin{pmatrix}
  \widetilde n(t) \\
  \widetilde r(t)
 \end{pmatrix}&:=
\begin{pmatrix}
 \frac {b+d_r+x^+}{x^+- x^-}\\
 \frac {a}{x^+- x^-}
\end{pmatrix}
 e^{t x^+}
-
\begin{pmatrix}
 \frac {b+d_r+x^-}{x^+- x^-}\\
 \frac {a}{x^+- x^-}
\end{pmatrix}
 e^{t x^-}, \\
 \begin{pmatrix}
  \bar n(t) \\
  \bar r(t)
 \end{pmatrix}&:=
\begin{pmatrix}
 \frac { b}{ \left(x^+- x^-\right) }
\\
 -\frac { b+d_r+ x^-}{x^+- x^- }
\end{pmatrix}
 e^{t x^+}
+
\begin{pmatrix}
- \frac {b}{\left(x^+- x^-\right) }
\\
 \frac { b+d_r+ x^+}{x^+- x^- }
\end{pmatrix}
 e^{t x^-}
\end{split}
\end{equation}
(remember that $b+d_r+ x^-<0$).
We note that $\lim_{t \to \infty} \bar n(t)=\lim_{t \to \infty} \widetilde n(t)=+\infty$; if, in addition,
$a>0$ 
then $\lim_{t \to \infty} \bar r(t)=\lim_{t \to \infty} \widetilde r(t)=+\infty$.
Hence
\begin{equation}\label{eq:prefirstmomentmatrix}
e^{At}=
\begin{pmatrix}
\widetilde n(t)  & \bar n(t) \\
\widetilde r(t) &   \bar r(t)\\
  \end{pmatrix}
\end{equation}
note that $e^{At}e^{As}=e^{As}e^{At}=e^{A(t+s)}$.


\section{Deterministic mass killing times}\label{sec:determ}

Between killing times, the bacterial population evolves randomly according to the rates \eqref{eq:rates},
each time starting from the set of survivors of the previous killing time.
We choose fixed time intervals $T_n=T$, where $T>0$; hence
mass killings occur at $S_n=n T$. We follow the strategy of \cite{cf:GaMaSch}.
 For all $n \ge 0$, we let the system evolve between $S_{n-1}$ and $S_n$ and we
count the number of survivors of each type at time $S_n$.
This is a 2-type branching process, whence we have survival if and only if
the Perron-Frobenius eigenvalue $\gamma_T^+$ of its first-moment matrix
\begin{equation}\label{eq:1stmomentmatrix}
 M(T):=\begin{pmatrix}
  (1-p)\widetilde n(T) & (1-p) \bar n(T)\\
\widetilde r(T) & \bar r(T)\\
 \end{pmatrix} \equiv
\begin{pmatrix}
  1-p & 0\\
0 & 1\\
 \end{pmatrix} e^{AT}
\end{equation}
satisfies $\gamma_T^+>1$
(see \cite{cf:BZ4}). Note that the entries of the $j$-th column of the matrix
are the average number of survivals after a mass killing at time $T$ starting from
one particle of type $j$ ($j=1$ being a susceptible particle, $j=2$ being a persistent
particle).

The following theorem holds for any (non-void) finite initial condition and the critical time $T_c(p)$ does not
depend on the initial condition (clearly it depends on all the parameters of the system,
even though, here, we emphasized only the dependence on $p$).

\begin{teo}\label{th:tcritical}
 Let $\lambda, a, d_n, d_r \ge 0$, $b >0$ such that equation~\eqref{eq:assumpdet} holds and $a+1-p>0$. 
For any $p>0$ there exists $T_c(p) \in (0, +\infty)$ such that the process dies out almost surely if and only if $T\le T_c(p)$.
Moreover $p \mapsto T_c(p)$ is a continuous, strictly increasing function and $\lim_{p \to 0} T_c(p)=0$.
\end{teo}

Requiring the inequality \eqref{eq:assumpdet} is quite natural, since if it does not hold, the bacterial
population would become extinct almost surely even without the action of the antibiotic.
Analogously, if $a+1-p=0$ (that is, $p=1$ and $a=0$) there cannot be survival since at the first killing time
the whole susceptible bacterial population is killed and the persistent population decreases (since
they cannot reproduce without switching to susceptible state and there is no switching back from susceptible to persistent).
If $p<1$ there can be survival even when switching form susceptible to persistent state is forbidden (that is, $a=0$).

Since $p \mapsto T_c(p)$ is increasing we have that if $t >T_c(1)$ then there is survival with
positive probability for all $p$, while if $t\in(0,T_c(1))$ there is a critical value $p_c(t) \in (0,1)$ such that
there is almost sure extinction if and only if $p \ge p_c(t)$. 

From the bacteria point of view, a winning strategy is to make $T_c$ as small as possible.
Here are three plots of $T_c$ where $a=b=0.01$ and $d_n=d_r=0.025$:
the first one on the left is the function $(\lambda,p) \mapsto T_c$,
the second one is the function $p \mapsto T_c$ (where $\lambda=3.7$) which show
a fast increasing of $T_c$ w.r.~to $p$ as $p$ is close to $1$. This means that if the antibiotic is slightly less
than perfectly efficient (that is, $p <1$) then the maximum admissible time interval $T_c$ to kill the bacterial population
is rapidly decreasing as $p$ decreases.
This is shown also by the third plot, which represents
the functions $\lambda \mapsto T_c$ for $p=1$ (red solid line), $p=0.95$ (yellow dot-dashed line) and
$p=0.9$ (green dashed line).

\begin{center}
\includegraphics[width=6.2truecm]{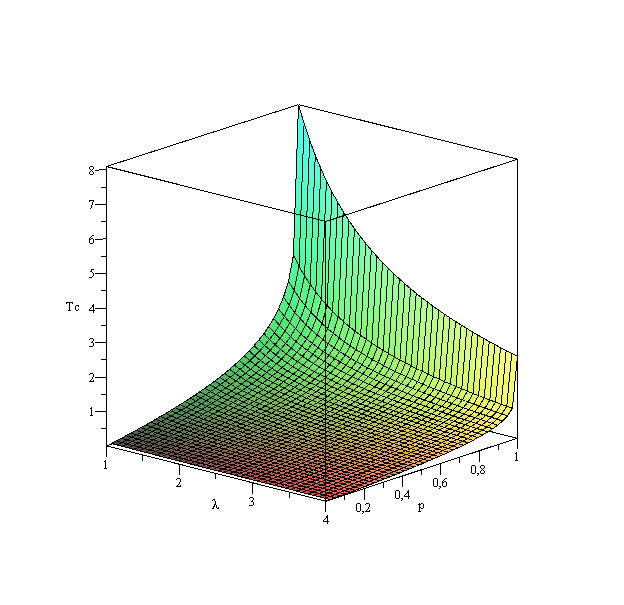}
\includegraphics[width=5truecm]{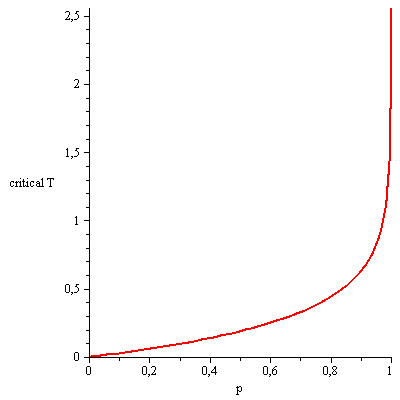}
\includegraphics[width=5truecm]{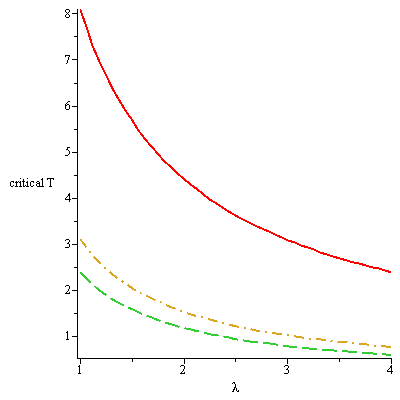}
\end{center}

Let us discuss briefly the behavior of $T_c$ w.r.~to $a$, in order to understand if $a>0$ is a better
strategy than $a=0$ from the bacteria point of view. First of all, when $p=1$ the only hope for survival for the
bacterial population is when $a>0$. On the other hand, of $p<1$ there might be a positive probability of survival
even if $a=0$, since a small fraction of susceptible bacteria may survive the action of the antibiotics. Hence in
this case, it is not trivial to decide whether $a>0$ is a better strategy than $a=0$ or not.
We note that, if $d_r>0$, then by equation~\eqref{eq:assumpdet} we have that $a \ge (\lambda-d_n)(b+d_r)/d_r$ implies
a.s.~extinction. More precisely one can prove that as $a \to (\lambda-d_n)(b+d_r)/d_r$ then $x^+ \to 0$ which
implies $T_c \to \infty$ (see the proof of Theorem~\ref{th:tcritical} for details) 
and eventually the situation becomes 
less favorable for the bacteria.

On the other hand, if we rewrite equation~\eqref{eq:systdiff2} as $h(x)=(x+b+d_r)(x-\lambda+d_n)+a(x+d_r)$ 
we see immediately that, when $d_r=0$, for every fixed $x>0$ (resp.~$x<0$) $h$ is strictly increasing (resp.~decreasing)
w.r.~to $a>0$ and $h(x) \to +\infty$ (resp.~$h(x) \to -\infty$) as $a \to \infty$. 
This implies immediately that $x^-$ and $x^+$ are strictly decreasing w.r.~to $a$ and that 
$x^- \to -\infty$ and $x^+ \to 0$ as $a \to \infty$. Hence,
when $d_r=0$, it is easy to prove that $T_c \to 0$ as $a \to \infty$
(this can be done by checking that, for every fixed $t>0$,
the function $F_{t,p}(1)$, introduced in the proof of Theorem~\ref{th:tcritical}, is negative for all $a$ sufficiently
large).
Here are two series of plots: on the left is $d_r=0$ and on the right is $d_r>0$ while the other parameters, with the exception
of $p$, are fixed: $p$ equals $0.9$ (green dashed line), $0.95$ (yellow dot-dashed line), $0.99$ (blue dotted line)
and $1$ (red solid line).

\begin{center}
 \includegraphics[width=6truecm]{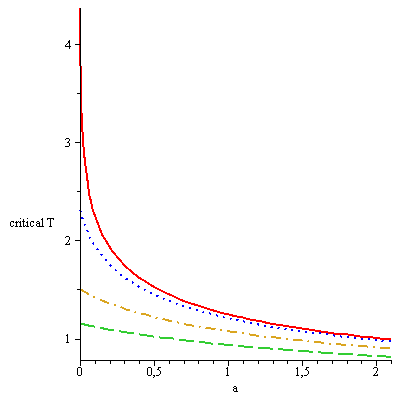}
\includegraphics[width=6truecm]{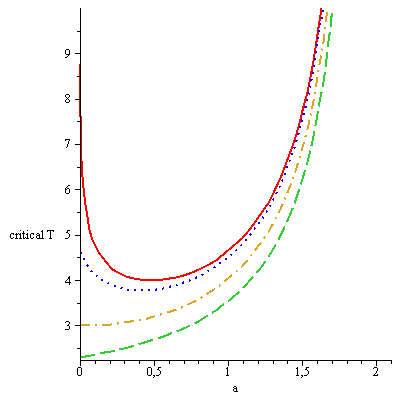}
\end{center}

We see that if $d_r=0$ then the best strategy for the bacterial population is to have $a$
as large as possible. If $d_r>0$ there is a critical value for $p$, above which, increasing
$a$ is a good strategy, up to a suitable value which minimizes $T_c$. Below the critical value
for $p$ the best situation for the bacterial population is $a=0$. 
As $a \to \infty$, eventually the situation becomes worse for the survival of the bacteria for all values
of $p<1$.

A more rigorous study of the behavior of $T_c$ with respect to $p$, $a$ and $b$ is possible
but it would exceed the aim of this paper.

\section{Random mass killing times}\label{sec:random}

%
%

%
%
We consider a sequence of i.i.d.~positive random times $\{T_n\}_{n \ge 1}$. According to
Lemma~\ref{lem:finiteintervals}, $\max\{n \colon S_n \le t\}<+\infty$ a.s.~for all $t>0$,
which means that there are a finite number of killing times in each finite interval almost surely.
Morever, suppose that the law of $T_n$ is $\mu_\beta$,
where $\{\mu_\beta\}_{\beta \in (0,+\infty)}$ is a family of
probability measures
on $(0, +\infty)$ (stochastically nondecreasing w.r.to $\beta$) satisfying

\begin{equation}\label{eq:assumpbeta}
\begin{split}
(1) & \ \displaystyle{ \forall t_0 >0,\  \lim_{\beta \to +\infty} \mu_\beta((0,t_0])=0;}\\
(2) & \ \forall \beta >0, \ \E_\beta:=\displaystyle{ \int t \, \mu_{\beta}(\diff t) <+\infty;} \\
(3) & \ \forall t_0>0, \ \displaystyle{ \lim_{\beta \to 0}  \int_{(0,t_0]} t \, \mu_\beta(\diff t)  \Big /
\int_{(t_0, \infty)} t \, \mu_\beta(\diff t)  =+\infty. } \\
\end{split}
\end{equation}

Clearly, since the family  $\{\mu_\beta\}_{\beta \in (0,+\infty)}$ is stochastically nondecreasing, we have that
$\beta \mapsto 
\int_{(t_0, \infty)} t \, \mu_\beta(\diff t)$ is a nondecreasing function for every $t_0 \ge 0$.
Moreover, $(3)$ implies
\[
 (4) \ 
\forall t_0 >0, \   \lim_{\beta \to 0} \mu_\beta((0,t_0]) =1. 
\]
As an example, consider the family of exponential laws $\mathrm{Exp}(1/\beta)$.

Roughly speaking, in this case, we have two randomizations, first we choose a realization of the random
sequence of times $\{T_n\}_{n \ge 1}$ (we call it, the \textit{environment})
and then we consider the random evolution of the population with
the chosen killing times.
More precisely, the sequence of snapshots of the system taken at the random times $\{S_n\}_{n \ge 0}$
is a multitype branching process in random environment (see \cite{cf:Tanny81} for the definition).
For each fixed $\beta$ we call this the $\beta$-process and each realization $\xi$ of the random time sequence $\{T_i\}_{n \ge 1}$
is our environment.
Henceforth, when we say that some event $\mathcal{A}$ 
(extinction or survival) has probability $0$ (resp.~$>0$)
for almost all realizations of the environment, we mean that the conditional probability of the event with respect to
the realization $\xi$ of the sequence of killing times is $0$ (resp.~$>0$) for almost all realizations $\xi$, that is,
$\pr(\xi \colon \pr(\mathcal{A}| T_i=\xi_i,\, \forall i \ge 1)=0)=1$ (resp.~$\pr(\xi \colon \pr(\mathcal{A}|  T_i=\xi_i,\, \forall i \ge 1)>0)=1$). 

Clearly if $\bar q(\xi)=(\bar q_1(\xi), \bar q_2(\xi))$ is the vector of extinction probabilities (starting from
one susceptible bacterium or from one persistent bacterium respectively), we have that $\pr(q(\xi)= \mathbf{1})$
is either $0$ or $1$. This means that there is a.s.~extinction for almost all realizations of the environment
or for almost no realizations of the environment.

Our results hold for any finite (non-void) initial condition
(and, again, the critical values depend on all the parameters of the system but not on
the initial condition). We assume again
the inequality \eqref{eq:assumpdet} to avoid spontaneous extinction of the bacterial population without the action
of the antibiotic. 

The first theorem states that if $\beta$ exceeds some finite critical value $\beta^1_c$ there is survival almost surely, that is,
for almost every realization of the environment. 
Roughly speaking, since the expected time is
nondecreasing with respect to $\beta$, it means that if the expected time is too large, the
action of the antibiotic might not be sufficient to kill the whole bacterial population.

\begin{teo}\label{th:survivalRT}
Let $\lambda, a, d_n, d_r \ge 0$, $b >0$  such that equation~\eqref{eq:assumpdet} holds. 
Let 
$\{\mu_\beta\}_{\beta \in (0,+\infty)}$ 
satisfy equation~\eqref{eq:assumpbeta} and $a+1-p>0$. If $\beta^1_c(p):=\sup\{\beta \in (0,+\infty)\colon 
\textrm{ the $\beta$-process dies out a.s.}\}$ then
$\beta^1_c(p)<+\infty$ and for all $\beta > \beta^1_c(p)$ we have survival with positive probability
for almost all realizations of the environment. Moreover $p \mapsto \beta_c^1(p)$ is nondecreasing.
\end{teo}

The second result tells us that if $\beta$ is smaller than some (strictly positive) critical value $\beta^2_c$
then, with probability $1$, the antibiotic will eventually kill the bacterial population. We show that,
$\beta^2_c$ tends to
$0$ as $p$ tends to $0$.

\begin{teo}\label{th:extinctionRT}
Let $\lambda, a, d_n, d_r \ge 0$, $b >0$  such that equation~\eqref{eq:assumpdet} holds. 
Let 
$\{\mu_\beta\}_{\beta \in (0,+\infty)}$ 
satisfy equation~\eqref{eq:assumpbeta}.
If $\beta^2_c(p):=\inf\{\beta \in (0,+\infty)\colon \textrm{ the $\beta$-process survives with positive probability}\}$ then
$\beta^2_c(p)>0$ such that for all $\beta < \beta^2_c(p)$ we have a.s.~extinction
for almost all realizations of the environment.
Moreover, 
 $p \mapsto \beta_c^2(p)$ is nondecreasing and $\inf_{p \to 0} \beta^2_c(p)=0$.
\end{teo}

Sharper results can be obtained if we assume that the random times have a exponential distribution with
expected value $1/\beta$. In this case there is a unique critical threshold $\beta_c$ separating
almost sure extinction from survival with positive probability. 

\begin{teo}\label{th:tcriticalrandom}
 Let $\lambda, d_n, d_r \ge 0$, $a,b >0$  such that equation~\eqref{eq:assumpdet} holds and $a+1-p>0$.. 
Let $\{\mu_\beta\}_{\beta \in (0,+\infty)}$ be a sequence of exponential laws 
$\mu_\beta \sim \mathrm{Exp}(1/\beta)$.
There exists $\beta_c(p) \in (0,+\infty)$ such that for all $\beta > \beta_c(p)$ we have survival with positive probability
for almost all realizations of the environment and for all $\beta 
< \beta_c(p)$ we have a.s.~extinction
for almost all realizations of the environment.
Moreover 
$p\mapsto \beta_c(p)$ is nondecreasing and
$\lim_{p \to 0} \beta_c(p)=0$.
\end{teo}

The next example shows, that for a generic $\{\mu_\beta\}_{\beta \in (0,+\infty)}$ satisfying our
hypotheses, we cannot always expect $\beta_c^2(p)=\beta_c^1(p)$. This means that the probability of survival
does not need to be monotone with respect to $\beta$.

\begin{exmp}\label{exmp:ex1}
 Let us take $\lambda=(\sqrt{21}+3)/4$, $b=(\sqrt{21}-3)/4$, $a=d_n=d_r=1/2$ and $p=1$. Since $p=1$ it is enough
to consider the expected size of the persistent population (at each killing time, susceptible bacteria
are killed). We note that $x^+=1$, $x^-=-1$ and $\bar r(t)=e^t(5-\sqrt{21})/8 + e^{-t}(3+\sqrt{21})/8$.
From equation~\eqref{eq:1stmoment2}, $\gamma^+_t=\bar r(t)$;
moreover the only strictly positive solution of $\bar r(t)=1$ is $T_c=\log(3+\sqrt{21})-\log(5-\sqrt{21})$.
We have $\bar r(t)<1$ for all $t\in (0,T_c)$ and $\bar r(t)>1$ for all $t >T_c$.
Consider the following family of measures
\[
 \mu_\beta:=
\begin{cases}
\frac12 \overline{\delta}_{\beta/10} + \frac12 \overline{\delta}_{3\max(\beta,1)} & \beta \in (0,15] \\
 \frac12 \overline{\delta}_{\beta-13.5} + \frac12 \overline{\delta}_{\beta-12} & \beta \in (15,+\infty) \\
\end{cases}
\]
where $\overline{\delta}_\alpha$ is the Dirac measure at $\alpha \in \mathbb{R}$.
Roughly speaking, for any fixed $\beta$, every time interval $T_i$ is chosen independently between two values
with probability $1/2$ each. It is straightforward to see that the family $\{\mu_\beta\}_{\beta \in (0,+\infty)}$
is stochastically increasing and satisfies equation~\eqref{eq:assumpbeta}.
According to \cite[Theorem 3.1]{cf:SW69} (see also Remark~\ref{rem:kingman} for details) if 
$\E[\log(\bar r(T_1))] \le 1$ there is a.s.~extinction for almost every realization of the environment,
while if  
$\E[\log(\bar r(T_1))] > 1$ there is positive probability of survival for almost every realization of the environment.
Here we have extinction if $\beta$ is close to $0$ (take for instance, $\beta=0.5$), we have survival
if $\beta=1$, we have extinction again if $\beta=15$ and we have survival if $\beta$ is large (take for instance,
$\beta=16.5$). Thus the probability of survival is not monotone and $\beta_c^2(p) \le 1<15 \le \beta_c^1(p)$.
\end{exmp}

\section{Proofs}\label{sec:proofs}

\begin{proof}[Proof of Theorem~\ref{th:tcritical}]
Define
\begin{equation}\label{eq:1stmoment2}
\begin{split}
 F_{t,p}(x)&:= x^2-x((1-p) \widetilde  n(t)+\bar r(t))+(1-p) (\widetilde n(t)\bar r(t)-\bar n(t) \widetilde r(t))\\
&=x^2 -x
\left (e^{tx^+}\left (1 -p \frac{b+d_r+x^+}{x^+- x^-} \right )
+ e^{tx^-}\left (1 + p \frac{b+d_r+x^-}{x^+- x^-} \right )
\right)
+(1-p)e^{t(x^++x^-)}\\
\end{split}
\end{equation}
where 
$x^++ x^-=\lambda-b-d_r-a-d_n$.
%
%
%
Let $\gamma^+_t \ge \gamma^-_t$ be the solutions of $F_{t,p}(x)=0$.
In order to check the inequality $\gamma^+_t>1$ we study the differentiable function $t \to F_{t,p}(1)$
for every fixed $p \in (0,1]$;
in particular we look for the solutions of the equation $F_{t,p}(1)=0$ with respect to $t$.
Clearly $\gamma^+_0=1$ and the other solution of $F_{0,p}(x)=0$ is $\gamma^-_0=1-p<1$.
Using equations~\eqref{eq:1stmoment2} and \eqref{eq:systdiff} we have
$\frac{\diff}{\diff t} F_{t,p}(1)|_{t=0}= 
p(d_r+b)>0$.
Hence there exists $\varepsilon >0$ such that $F_{t,p}(1)>0$ for  all $t\in (0,\varepsilon)$; thus,
by continuity and since $\gamma^-_0<1$,
we have that
$\gamma_t^+<1$ for all $t\in (0,\varepsilon)$.
Since $\lim_{t \to \infty} F_{t,p}(1)=-\infty$ for all $p \in (0,1]$, there is at least one strictly positive
solution to $F_{t,p}(1)=0$ w.r.~to $t$. In order to show that it is unique,
observe that, since $x^-<0$,
\[
 \frac{\diff}{\diff t} F_{t,p}(1)= e^{t(x^++x^-)}
 \left [ (1-p)(x^++x^-)
 -\left ( x^+ e^{t|x^-|}\left (1 -p h \right )
- |x^-| e^{-tx^+}\left (1 + p (1-h) \right )
\right)
\right ]
\]
where $h={(b+d_r+x^+)}/{(x^+- x^-)}$. 
Clearly $\sgn(\frac{\diff}{\diff t} F_{t,p}(1))=\sgn(L(t,p))$ where $L(t,p):=F_{t,p}(1)e^{-t(x^++x^-)}$
%
Since $x^+>0$ we have that (for every fixed $p$) $t \mapsto L(t,p)$ is a strictly decreasing 
function such
that $L(0,p)>0$ and $\lim_{t \to +\infty} L(t,p)=-\infty$; thus there
exists a unique $T_c=T_c(p) \in (0, +\infty)$ such that
$F_{t,p}(1)>0$ (resp.~$F_{t,p}(1)<0$) if $t \in (0, T_c)$
(resp.~$t \in (T_c, \infty)$). This implies that
 $\gamma^+_t<1$ for all $t \in (0,T_c)$ and
$\gamma^+_t \ge 1$ for all $t \in [T_c, +\infty)$ (clearly $F_{T_c(p),p}(1)=0$ and $\gamma^+_{T_c(p)}=1$).

By elementary analysis $p \mapsto T_p$ is a differentiable
function (for every fixed $t \ge 0$).
Moreover, by convexity, since $(b+d_r+x^+)/(x^+- x^-) \in [0,1]$,
\[
\begin{split}
 \frac{\diff}{\diff p} F_{t,p}(1)&=e^{tx^+} \frac{b+d_r+x^+}{x^+- x^-}
+ e^{tx^-} \left (- \frac{b+d_r+x^-}{x^+- x^-} \right )
-e^{t(x^++x^-)}\\
& \ge e^{t \left ( x^+(b+d_r+x^+)/(x^+- x^-)-x^-(b+d_r+x^-)/(x^+- x^-) \right )}-e^{t(x^++x^-)} \\
&=
e^{t(x^++x^-+b+d_r)}-e^{t(x^++x^-)}>0
\end{split}
\]
for all $t>0$. Hence $p \mapsto F_{t,p}(1)$ is strictly increasing which implies that $p \mapsto T_c(p)$ is strictly increasing.
Since $\lim_{p \to 0} F_{t,p}(1)=F_{t,0}(1)<0$ for all $t>0$ (indeed the process is supercritical in the absence of mass-killing),
we have that $\lim_{p \to 0}T_c(p)=0$.

Finally, if $a>0$ then there is survival starting from $1$ persistent particle if and only if there is survival starting
from $1$ susceptible particle; thus, since the process is monotone with respect to the initial state, the
long-term behavior is the same as long as the initial state is finite. If $a=0$ then $p<1$ and the Perron-Frobenius eigenvalue
$x^+=m_{11}(t)$, hence our result holds starting from $1$ susceptible particle; nevertheless, since $b>0$,
even if we start from $1$ persistent particle there is a positive probability it becomes a susceptible one, hence
there is a positive probability of survival starting from $1$ susceptible particle if and only if there is a positive
probability of survival starting from any finite initial state.
\end{proof}

The proof of the following Lemma is very easy, nevertheless we include it for
completeness.

\begin{lem}\label{lem:finiteintervals}
 Let $\{T_i\}_{i \in \N}$ be  nonnegative i.i.d.~random variables. If $\pr(T_1>0)>0$
 then $\E[N_t]<+\infty$ where $N_t:=\max\{n\colon \sum_{i=0}^n T_i \le t\}$.
\end{lem}

\begin{proof}
 Let $S_n:=\sum_{i=0}^n T_i$ and suppose that $\E[T_i^4]<+\infty$: in this case define
 $\E[T_i]=:\mu >0$, $\E[(T_i-\mu)^2]=:\sigma^2$ and $\E[(T_i-\mu)^4]=:r^4$.
 Clearly, eventually as $n \to \infty$,
 \[
  \begin{split}
\pr(N_t \ge n)&=\pr(S_n \le t)= \pr(S_n/n-\mu \le t/n-\mu)\\
&\le \pr(|S_n/n-\mu| \ge \mu/2) \le \E[|S_n/n-\mu|^4]/(\mu/2)^4=
\frac{16}{n^4 \mu^4} \E \left [(\sum_{i=1}^n (T_i-\mu))^4 \right
]=(*).
  \end{split}
 \]
Now $(\sum_{i=1}^n (T_i-\mu))^4=\sum_{\textbf{i} \in \{1, \ldots, n\}^4} \prod_{j=1}^4 (T_{\textbf{i}_j}-\mu)$;
moreover the independence of $\{T_i\}_{i \in \N}$ yields
 $
\E[\prod_{j=1}^4 (T_{\textbf{i}_j}-\mu)]=0
$
if there exists $j$ such that $\textbf{i}_j \not = \textbf{i}_k$ for all $k \not = j$.
Hence
 \[
  \begin{split}
 (*)
 =   \frac{16}{n^4 \mu^4} \E \left [\sum_{j=1}^n (T_{\textbf{i}_j}-\mu)^4 + \sum_{h,j \colon h \not = j} (T_{\textbf{i}_h}-\mu)^2
 (T_{\textbf{i}_j}-\mu)^2 \right ] 
= \frac{16}{n^4 \mu^4} [n r^4 +n(n-1) \sigma^4] 
\le C/n^2
  \end{split}
 \]
thus 
$\E[N_t]= \sum_{n
\in \N} \pr(N_t \ge n) <+\infty$.

In the general case, define $\overline T_i := \min(T_i, 1)$. Then $\E[\overline T_i^4]<+\infty$ hence
$\E[N_t] \le \E[\bar N_t]<+\infty$ where $\bar N_t:=\max\{n\colon \sum_{i=0}^n \bar T_i \le t\}$.
\end{proof}

\begin{rem}\label{rem:kingman}
In the random killing time case we deal, in general, with a multitype branching process in random environment
where the sequence of environments is i.i.d. hence, if we denote
by $M_n:=M(T_n)$ the first-moment matrix \eqref{eq:1stmomentmatrix} with $T=T_n$,
by Kingman Subadditive Theorem, we have (see for instance
\cite{cf:AthKarlin71, cf:FurstKest60, cf:Kingman73, cf:Steele89, cf:Tanny81})
\[
 \lim_{n \to \infty} n^{-1} \log \left ( \left \| \prod_{i=n}^1 M_i \right \| \right )=\delta_\beta \equiv \E[\delta_\beta],
\quad \textrm{a.s.}
\]
and $\delta_\beta= \lim_{n \to \infty} n^{-1} \E \left [\log \left ( \left \| \prod_{i=n}^1 M_i \right \| \right ) \right ]$
where $\|M\|:= \max_j \sum_i |M_{ij}|$ and $\prod_{i=n}^1 M_i:=M_n M_{n-1} \cdots M_1$.
%
%
%
This plays the role of the Perron-Frobenius eigenvalue of the deterministic case
and it will be useful in the next proofs (where we use \cite[Teorems 9.6 and 9.10]{cf:Tanny81}; in the case $p=1$ one
may use also \cite[Theorem 3.1]{cf:SW69} instead).

Moreover, it is easy to check that the conditions of \cite[Teorems 9.10]{cf:Tanny81} are satisfied.
Indeed, the entries of the first-moment matrix satisfy $m_{i,j}(t)>0$ for all $t>0$ and for all $i,j=1,2$. Moreover,
since $\mu_\beta((0,+\infty))=1$ for all $\beta \in (0,+\infty)$, we have
$\pr(\min_{i,j} (M_1)_{i,j}>0)=1$. Finally, if we start with a susceptible particle then
$\pr(N(t)\ge 1) \ge e^{-(a+d_n)t}$ hence
\[
 \E[|\log (1-\pr(N(t)=0))|] \le \int_{(0,+\infty)} (a+d_n)t \mu_\beta(\diff t)
<+\infty, \quad \forall \beta \in (0,+\infty).
\]
On the other hand, if the initial condition is a persistent particle we proceed by using $R(t)$ instead of $N(t)$.
One can check analogously that our branching process in random environment is strongly regular
(see \cite[Definition 9.1]{cf:Tanny81}).
Hence, according to \cite[Theorem 9.10]{cf:Tanny81}, we have:
\begin{enumerate}
\item $\delta_\beta \le 0$ implies a.s.~extinction for almost all realizations of the environment,
\item $\delta_\beta >0$ implies survival with positive probability for almost all realizations of the environment.
\end{enumerate}
Clearly the probability of survival is $0$ if and only if the conditional probability of survival is $0$ 
for almost all realizations of the environment. On the other hand, since $\pr(q(\xi)= \mathbf{1})$
is either $0$ or $1$ (see Section~\ref{sec:random}), then the probability of survival is strictly positive if and only if 
the conditional probability of survival is strictly positive 
for almost all realizations of the environment.
\end{rem}

\begin{proof}[Proof of Theorem~\ref{th:survivalRT}]
First of all we check the integrability condition, that is, for all $n \ge 1$,
\[
\E \Big  [n^{-1} \Big |\log \Big ( \Big \| \prod_{i=n}^1 M_i \Big \| \Big ) \Big | \Big ]\equiv
\int n^{-1} \Big |\log \Big ( \Big \| \prod_{i=n}^1 M(t_i) \Big \| \Big ) \Big | \prod_{i=1}^n \mu_\beta(\diff
t_i)<+\infty,
\]
where $\prod_{i=1}^n \mu_\beta$ is a probability product measure on $\mathbb{R}^n$.
Below we show that if $p<1$ then
$\left \| \prod_{i=n}^1 M_i \right \| \ge (1-p)^n
\varepsilon^n$ (for some $\varepsilon >0$) while
if $a>0$ then $\left \| \prod_{i=n}^1 M_i \right \| \ge \varepsilon^n$ (for some $\varepsilon >0$).
Hence if $a+1-p>0$ , for some $\varepsilon^\prime>0$,
\[
\int n^{-1} \log^- \Big ( \Big \| \prod_{i=n}^1 M(t_i) \Big \| \Big ) \prod_{i=1}^n \mu_\beta(\diff
t_i) \le \log^-(\varepsilon^\prime)<+\infty
\]
since $\log^-$ is nonincreasing (where $\log^-(\cdot):=\max (0, -\log(\cdot))$).
Thus we just need to prove that \break
$\int n^{-1} \big |\log^+ \big ( \big \| \prod_{i=n}^1 M(t_i) \big \| \big) \big | \prod_{i=1}^n \mu_\beta(\diff
t_i)<+\infty$
(where $\log^+(\cdot):=\max (0, \log(\cdot))$). From equation~\eqref{eq:fundsolutions} we have
$\|M(t)\| \le K e^{tx^+}$; hence, since $log^+$ is nondecreasing, $\big \| \prod_{i=n}^1 M_i \big \| \le
\prod_{i=1}^n \left \|  M_i \right \|$ and the expected value of $\mu_\beta$ is finite (for all $\beta$)
we have $\int n^{-1} \big |\log^+ \big ( \big \| \prod_{i=n}^1 M(t_i) \big \| \big ) \big | \prod_{i=1}^n \mu_\beta(\diff
t_i)<+\infty$.

Suppose $p<1$ and $a>0$.
For all $t, \tau >0$ there exists $\beta_0(\tau, t)$ such that $\mu_\beta([t, +\infty))> 1-\tau$
for all $\beta > \beta_0(\tau, t)$.
By continuity and compactness we have that, for some $\varepsilon >0$,
\[
 M(t) \ge \begin{pmatrix}
           (1-p) \varepsilon & 0 \\
0 & \varepsilon
          \end{pmatrix}=:\overline M_0, \quad \forall t \ge 0
\]
where, by definition, $A \ge B$ if and only if $A_{ij} \ge B_{ij}$.
It is easy to show, by using equation~\eqref{eq:fundsolutions}, that there exists $t_p \in [0, +\infty)$ such that
\[
 M(t) \ge \begin{pmatrix}
           \frac{4}{\varepsilon(1-p)}  & 0 \\
0 & \frac{4}{\varepsilon}
          \end{pmatrix}=:\overline M_1, \quad \forall t \ge t_p.
\]
Let $\beta > \beta_0(1/2, t_p)$ and $\{M_i\}_{i \ge 1}$ the corresponding sequence
of random first-moment matrices ($M_i:=M(T_i)$); 
thus, according to the Law of Large Numbers,
with probability 1, as $n \to \infty$, $\#\{i \le n \colon M_i \ge \overline M_1 \} \ge n/2$ which implies that
$\prod_{i=n}^1 M_i \ge \overline M_0^{n/2}\, \overline M_1^{n/2} = 2^n \ident$ almost surely.
Hence $\liminf_{n \to \infty} n^{-1} \log \left ( \left \| \prod_{i=n}^1 M_i \right \| \right ) \ge \log 2 >0$
a.s.~which, according to \cite[Theorems 9.10]{cf:Tanny81}, implies survival with positive probability for almost every
realization of the environment. Hence, by definition, $\beta_1(p) \le \beta_0(1/4,t_p)$.

The usual coupling technique shows that for any fixed choice of the parameters $\lambda$, $a$, $b$,
$d_n$, $d_r$ and for any realization of the environment,
the probability of survival is nonincreasing with respect to $p$. Hence $p \mapsto \beta_c^1(p)$ is
nondecreasing.

If $p=1$ then $a>0$ and we are dealing essentially with a single
population (the persistent bacteria as in \cite{cf:GaMaSch}),
since after each killing time we have just persistent bacteria
left. The first moment, starting with a susceptible bacterium, is
$\bar r(t)$. The proof is essentially the same since $\bar r(T)
\ge \varepsilon>0$ for all $t \ge 0$ and $\bar r(T) \ge
4/\varepsilon$ for all $t \ge t_p$. The result follows from
\cite[Theorem 3.1]{cf:SW69}.
Since $a>0$ then
there is survival starting with a persistent bacterium if and only if there is survival starting with a susceptible
one.

Finally if $a=0$ (hence $p<1$) again we are dealing essentially with a single population: the susceptible bacteria.
The first moment, starting with a susceptible bacterium,  is $\widetilde n$ and the result follows (from  \cite[Theorem 3.1]{cf:SW69}
as before) from the inequalities
$\widetilde n(t) \ge (1-p)\varepsilon$ for all $t \ge 0$ and $\widetilde n(t) \ge 4/((1-p)\varepsilon)$ for all $t \ge t_p$.
\end{proof}

\begin{proof}[Proof of Theorem~\ref{th:extinctionRT}]
If $p=1$ and $a=0$ the process becomes extinct almost surely. We suppose henceforth that $a+1-p>0$.
Since $\log$ is increasing and $\{M_i\}_{i \ge 1}$ are identically distributed, 
we have
\[
\E \Big  [n^{-1} \log \Big ( \Big \| \prod_{i=n}^1 M_i \Big \| \Big ) \Big ] \le
n^{-1} \E \Big  [\log \Big ( \prod_{i=n}^1 \Big \|  M_i \Big \| \Big ) \Big ] \le
\log (\E[\|M_1\|]).
\]
Thus, if we prove that $\log (\E[\|M_1\|])<0$ for every sufficiently small $\beta$ then \cite[Theorem 9.6]{cf:Tanny81}
guarantees a.s. extinction for almost all configurations
(if $p=1$ one can also use \cite[Theorem 3.1]{cf:SW69} instead).

It is straightforward to show that $t \mapsto \|M(t)\|$ is 
differentiable from the right at $0$.
By elementary computations, since $M(0)=\ident$, $\frac{\diff}{\diff t}  \log \|M(t)\|\Big |_{t=0}=
\frac{\diff}{\diff t} \|M(t)\| \Big |_{t=0} 
=:-m<0$.
Hence, there exists $t_0>0$ such that $\log (\|M(t)\|)  \le -t m/2$ for all $t \in [0,t_0]$.
On the other hand, $\log(\|M(t)\|) \le C x^+ t$ for all $t>0$ and
some $C>0$.
 Finally by equation~\eqref{eq:assumpbeta},
\[
\begin{split}
 \int \log(\|M(t)\|) \mu_\beta(\diff t) &= \int_{(0,t_0]} \log(\|M(t)\|) \mu_\beta(\diff t) +
\int_{(t_0, \infty)} \log(\|M(t)\|) \mu_\beta(\diff t) \\
& \le -\int_{(0,t_0]} \frac{tm}{2} \mu_\beta(\diff t) +
\int_{(t_0, \infty)} Cx^+ t \mu_\beta(\diff t) \\
& \le -\int_{(0,t_0]} \frac{tm}{2} \mu_\beta(\diff t) \left ( 1 -
\frac{\int_{(t_0, \infty)}  t \mu_\beta(\diff t)}{\int_{(0,t_0]} t \mu_\beta(\diff t)} \frac{2Cx^+}{m} \right )<0 \\
\end{split}
\]
for every sufficiently small $\beta$. Hence $\beta_2(p)>0$.

As in the proof of Theorem~\ref{th:extinctionRT}, for every realization of the environment,
the probability of survival is nonincreasing with respect to $p$. Hence $p \mapsto \beta_c^2(p)$ is
nondecreasing.

Let us fix $\beta >0$.
It is well-known that
\[
\lim_{n \to \infty} n^{-1} \log \Big ( \Big \| \prod_{i=n}^1 e^{A T_i} \Big \| \Big )
=\lim_{n \to \infty} n^{-1} \log \Big ( \Big \| e^{A \sum_{i=1}^n T_i} \Big \| \Big )
= 
x^+ \E_\beta
>0, \quad a.s.
\]
($e^{x^+}$ is the maximum eigenvalue of $e^A$). 
Moreover
$M(t) \ge (1-p) e^{At}$ thus $\prod_{i=n}^1 M_i \ge (1-p)^n e^{A \sum_{i=1}^n T_i}$ and
\[
\lim_{n \to \infty} n^{-1} \log \left ( \left \| \prod_{i=n}^1 M_i \right \| \right )
\ge \log (1-p) + 
x^+ \E_\beta
\quad a.s.
\]
which is eventually strictly positive, as $p \to 0$. Hence, according to
\cite[Theorem 9.10]{cf:Tanny81} the process eventually survives as $p \to 0$;
thus,  by definition, $\beta_2(p) \le \beta$ eventually as $p \to 0$.

\end{proof}

\begin{proof}[Proof of Theorem~\ref{th:tcriticalrandom}]
We use a modification of the construction shown in \cite{cf:GaMaSch}.
Given $\beta_1 \ge \beta_2$, it is well-known that, by using the classical decimation procedure,
it is possible to construct two sequences $\{\overline{T}_i^1\}_{i \ge 1}$
and $\{\overline{T}_i^2\}_{i \ge 1}$ in such a way that, for every trajectory, $\{\overline{T}^1_i(\omega) \colon i \ge 1\} \subseteq
\{\overline{T}^2_i(\omega) \colon i \ge 1\}$.

Consider the binary tree $\mathbb{T}$ whose vertices are the set $V$ of finite words of the alphabet $\{0,1\}$ and whose root
is the empty word $\emptyset$. Every nonempty word  $v_1 v_2 \ldots v_n$ is connected to its parent
$v_1 v_2 \ldots v_{n-1}$ and its two children $v_1 v_2 \ldots v_n \, 0$ and $v_1 v_2 \ldots v_n \, 1$
(the root is connected to $0$ and $1$).

To each vertex $v$ corresponds a variable $S_v\sim \mathrm{Exp}(\lambda)$ which represents the time interval between
its birth and its splitting (when it gives birth). We assume that $\{S_v\}_{v \in V}$ is an i.i.d.~family of random variables.
Define $T_\emptyset:=0$ and, for every nonempty word $v=v_1 \ldots v_n$, $T_v= \sum_{i=1}^{n-1} S_{v_1 \ldots v_i}$.
Consider now the tree $\widehat{\mathbb{T}}$ on $\mathbb{T} \times [0, +\infty)$ as follows: the set of vertices is
$\widehat V:=\{(v, T_v), (v, T_v+S_v)\colon v \in \mathbb{T}\}$.
We have vertical edges between $(v, T_v)$ and $(v, S_v+T_v)$ (for all $v \in V$); we have horizontal edges between
$(v, T_v+S_v)$ and
each of its two children
$(vw, T_v+S_v)$ where $w \in \{0,1\}$ (for all $v \in V$).
The vertical edge between $(v,T_v)$ and $(v,T_v+S_v)$
represents the time interval between the birth of the particle $v$ and its splitting time. The horizontal edge
between $(v,Tv)$ and $(v1,Tv)$ represent the birth of a child of $v$ while we consider the other particle, namely $v0$,
as $v$ itself after giving birth.

Independently of everything constructed so far, we consider four independent families of Poisson point processes $\{W^1_v\}_{v \in V}$,
$\{W^2_v\}_{v \in V}$, $\{D^1_v\}_{v \in V}$ and $\{D^2_v\}_{v \in V}$ on $[0, +\infty)$
with intensities $b$, $a$, $d_n$ and $d_r$ respectively. We color the tree in white (susceptible state),
red (persistent state) and black (dead particle) as follows: we start with a white vertex $(\emptyset, 0)$ and we extend the color
to the branches along the timeline until
we reach a point of one of the Poisson processes. If we meet a $D^1_v$ point and the current color is white we switch to black and there
are not modifications anymore in that subtree along the timeline (death of the particle),
the same happens if we meet a $D^2_v$ point and the current color is red.
If the current color is not black, then everytime we meet
a $W^2_v$ point we switch to red and everytime we meet a $W^1_v$ point we switch to white.
Black color is not modified when we meet new points in the Poisson processes.
At every split point if the current color
is white or black then we use the same color for the horizontal edges and we continue starting from the two new vertices. If the color is
red then we use the same color for the horizontal edge which connects to the child whose name ends with $0$ (and we start again from
there with a red vertex); we switch to black
for the horizontal edge connecting to the other son (hence, the whole subtree branching from this vertex is black,
since red (persistent) particles do not reproduce).

So far we modelled the natural evolution of the system; now we add
the action of the antibiotics. To this aim, independently again,
we add the coupled Poisson processes $\{\overline{T}_i^j\}_{i \ge
1}$ ($j=1,2$) defined above and we consider two independent
families of independent Bernoulli variables $\{B^1_e\}_{e \in
V\widehat{\mathbb{T}}}$ with parameter $p_1$ and $\{B^1_e\}_{e \in
V\widehat{\mathbb{T}}}$ with parameter $p_2$ where $p_1 \le p_2$
and $V\widehat{\mathbb{T}}$ is the set of vertical edges of
$\widehat{\mathbb{T}}$. There is an analogous decimation procedure
which allows to couple these two Bernoulli processes in such a way
that if $B^1_e=1$ then $B^2_e=1$. At each time $T^1_i$ (resp.~$T^2_i$)
we consider all the white vertical edges intersecting the
horizontal plane with time coordinate $T^1_i$ (resp.~$T^2_i$); for all
such edges $e$, if $B^j_e=1$ we switch to black in the
corresponding $j$ model (hence the whole subtree is black),
otherwise nothing happens.

In each model, at any time $t \ge 0$, let $N_t$ (resp.~$R_t$) be
the number of white (resp.~red) vertical edges which intersects
the horizontal plane with time coordinate $t$; $\{(N_t, R_t)\}_{t
\ge 0}$ is the formal definition of the process. It is clear that
the white/red edges in the second model have the same color in the
first one, hence the non-black portion of the tree in the second
model is a subtree of the non-black portion of the tree in the
first model. In this construction, the event of survival is the collection
of all the trees which have at least a red/white branch intersecting
the horizontal plane $t$ for every $t>0$.
This implies easily that the probability of
survival of a model $(\beta_1,p_1)$ is larger or equal than the
probability of survival of a model $(\beta_2, p_2)$ (where
$\beta_1 \ge \beta_2$ and $p_2 \ge p_1$). 
More precisely, we coupled the environments in such a way that
the conditional (w.r.~to the environment) probabilities of survival of the model $(\beta_1,p_1)$ are larger or equal than the
conditional (w.r.~to the coupled environment) probabilities of survival of the
 model $(\beta_2, p_2)$.
In particular, for any
fixed $p$, the probability of survival is nondecreasing w.r.~to
$\beta$ and, for any fixed $\beta$, is nonincreasing with
respect to $p$. 
%
%
%

If we define
$\beta_c(p):=\inf \{\beta >0\colon \textrm{ the process survives with positive probability}\}$ then 
$\beta_c(p)>0$ (according to Theorem~\ref{th:extinctionRT}).
Since the probability of survival is  nondecreasing w.r.~to $\beta$, we have that for all $\beta>\beta_c(p)$ there is survival
with positive probability
and for all $\beta < \beta_c(p)$ we have almost sure extinction.
Hence $\beta_c(p)=\sup \{\beta >0\colon \textrm{ the process dies out almost surely}\}$, thus 
$\beta_c(p)<+\infty$
(according to Theorem~\ref{th:survivalRT}).
Clearly $\beta_c(p)$ is nondecreasing with respect to $p$ and,
according to Theorem~\ref{th:extinctionRT}, $\lim_{p \to 0} \beta_c(p)=0$.

\end{proof}

%
%

\end{document}